\documentclass[a4paper,11pt]{article}
\usepackage{amsthm,amsmath,amssymb}
\usepackage{tikz}
\usepackage{graphicx,subcaption,caption}
\usepackage{algorithm,algorithmic}
\usepackage{pgfplots}
\usepackage{multirow}

\usepackage{enumerate}

\newtheorem{theorem}{Theorem}[section]
\newtheorem{lemma}[theorem]{Lemma}
\newtheorem{problem}[theorem]{Problem}
\newtheorem{assumption}[theorem]{Assumption}

\newtheorem{corollary}[theorem]{Corollary} 
\theoremstyle{definition}
\newtheorem{definition}[theorem]{Definition}

\theoremstyle{remark}
\newtheorem{remark}[theorem]{Remark}

\begin{document}
	\title{Uncomputability of Global Optima for Nonconvex Functions in the Oracle Model}
	\author{K. Lakshmanan \\ Department of Computer Science and Engineering, \\ Indian Institute of Technology (BHU) Varanasi, 221005, India, \\ Email: lakshmanank.cse@iitbhu.ac.in}
	\date{}

	\maketitle
	
	\begin{abstract}
		While it is well known that finding approximate optima of non-convex functions is computationally intractable, we show that the problem is, in fact, uncomputable in the oracle model. Specifically, we prove that no algorithm with access only to a function oracle can compute the global minimum or even an $\epsilon$-approximation of the minimizer or minimal value. We then characterize a necessary and sufficient condition under which global optima become computable, based on the existence of a computable predicate that subsumes the global optimality condition. As an illustrative example, we consider the basin of attraction around a global minimizer as such a property and propose a simple algorithm that converges to the global minimum when a bound on the basin is known. Finally, we provide numerical experiments on standard benchmark functions to demonstrate the algorithm’s practical performance. \\

		\textbf{Keywords:}  Non-convex functions, Oracle model,
		Computability, Global minimum, Turing machine, Basin of attraction, Gradient descent
		
		\textbf{Mathematics Subject Classification 2020: 68Q17, 90C26}
	\end{abstract}

	\section{Introduction and Preliminaries}
	
	The problem of computing the global minimum of a non-convex continuous function \( f : C \to \mathbb{R} \), where \( C \subset \mathbb{R}^d \) is a closed, compact subset, is central to optimization theory \cite{globook}. A global minimum is a point \( x^* \in C \) such that \( f(x^*) \leq f(x) \) for all \( x \in C \). By the extreme value theorem, this minimum is always attained for continuous \( f \) on compact \( C \), and our goal is to identify at least one such point.
	
	In this paper, we analyze this problem in the \emph{oracle model}, where function values are accessible only via queries to a black-box oracle. This model differs from the computable real-function setting studied in computable analysis (e.g., \cite{pour}), and is more general than the restricted oracle models considered in works such as \cite{optmzer}, where the optimal \emph{value} may be computable, but the \emph{optimizer} is not. In contrast, we show that \emph{neither} the optimizer nor the optimal value is computable in the oracle model we consider.
	
	It is known that global minimization of non-convex functions is NP-hard, even for certain continuous functions, via reductions from problems such as subset-sum \cite{nph}. Let us define the \(\epsilon\)-optima set as \( S = \{x \in C \mid |f(x) - f(x^*)| \leq \epsilon \} \). In optimization literature \cite{foster2019complexity,zhang2020complexity}, even approximating the global minimum or reaching \(\epsilon\)-stationary points is known to be intractable. In this paper, we go further and show that the set \( S \) is in fact \emph{non-computable}. This is a strictly stronger result than intractability — it implies no algorithm, regardless of resources, can compute an \(\epsilon\)-optimal solution in the oracle setting.
	
	While it is easy to see that a simple grid search produces a sequence converging to the global minimum, such a method requires an exponential number of oracle calls for Lipschitz continuous functions, even when the Lipschitz constant is known \cite{nest}. We show that having access to a Lipschitz constant (or any property bounding function variation) is essential for computability — and without such knowledge, global optimization becomes non-computable.
	
	\paragraph{Our main contributions are as follows:}
	\begin{enumerate}
		\item We formalize an oracle-based model of optimization that applies to a broad range of real-world non-convex problems and generalizes existing oracle models.
		\item We prove that, in this oracle setting, both the minimal value and any \(\epsilon\)-minimizer of a non-convex continuous function are non-computable.
		\item We identify a natural property that a function must satisfy if its global optima are to be computable or even approximable.
		\item We introduce the notion of a “basin of attraction” as a computable property, and show that if this is known, a convergent algorithm to the global optimum is possible.
	\end{enumerate}

	While the nonconvex optimization problem has long been known to be intractable, especially under worst-case complexity assumptions, the boundary between intractability and uncomputability remains underexplored. Our work aims to make this boundary precise by examining the problem through the lens of oracle-based computation. The results not only illuminate fundamental limits but also justify why certain classes of heuristics (such as gradient-based methods) are inherently limited.

	\section{Turing Machines, Computable Numbers, and Finite Precision Representation}
	
	We begin by recalling the standard definition of a Turing machine.
	
	\begin{definition}
		A Turing machine consists of three infinite tapes divided into cells, a reading head which scans one cell of each tape at a time, and a finite set of internal states \( Q = \{q_0, q_1, \ldots, q_n\} \), with \( n \geq 1 \). Each cell contains either a blank symbol or the symbol 1. In a single step, the machine may simultaneously:
		\begin{enumerate}
			\item change its internal state;
			\item overwrite the scanned symbol \( s \) with a new symbol \( s' \in S = \{1, B\} \);
			\item move the reading head one cell to the left (L) or right (R).
		\end{enumerate}
		The operation is governed by a partial function \( \Gamma : Q \times S^3 \rightarrow Q \times (S \times \{L, R\})^3 \).
	\end{definition}
	
	\begin{remark}
		The map \( \Gamma \), viewed as a finite set of quintuples, is referred to as a Turing program. For a tuple \( (q, s_1, s_2, s_3, q', s_1', X_1, s_2', X_2, s_3', X_3) \in \Gamma \), the interpretation is: when in state \( q \), reading symbols \( s_1, s_2, s_3 \), the machine transitions to state \( q' \), writes \( s_i' \) on tape \( i \), and moves the head in direction \( X_i \) (L or R) on tape \( i \), for \( i = 1,2,3 \).
	\end{remark}
	
	\subsection{Computable Numbers}
	
	\begin{definition}
		A function \( f : \mathbb{N} \to \mathbb{N} \) is said to be \emph{computable} if there exists a Turing machine that, on input \( x \in \mathbb{N} \), writes \( f(x) \) on its output tape.
	\end{definition}
	
	All rational numbers are assumed to be computable. The standard definition of a computable real number (e.g., \cite{optmzer}) is as follows:
	
	\begin{definition}
		A sequence of rational numbers \( (r_n)_{n \in \mathbb{N}} \) is called \emph{computable} if there exist recursive functions \( s, p, q : \mathbb{N} \rightarrow \mathbb{N} \) such that
		\[
		r_n = (-1)^{s(n)} \cdot \frac{p(n)}{q(n)}
		\]
		for all \( n \in \mathbb{N} \).
	\end{definition}
	
	\begin{definition}
		A sequence \( (x_n) \subset \mathbb{R} \) is said to \emph{converge effectively} to a limit \( x \in \mathbb{R} \) if
		\[
		|x_n - x| \leq 2^{-n}
		\]
		for all \( n \in \mathbb{N} \).
	\end{definition}
	
	\begin{definition}
		A real number \( x \in \mathbb{R} \) is called \emph{computable} if there exists a computable rational sequence \( (r_n) \) that converges effectively to \( x \). The sequence \( (r_n) \) is then said to be a \emph{representation} of \( x \).
	\end{definition}
	
	The set of computable numbers is closed under basic arithmetic operations and effective convergence.
	
	\subsection{Finite Precision Representation}
	
	To formalize computability in practice, we consider finite-precision representations. This models real-world computation and is not a restrictive assumption.
	
	Let \( x \in \mathbb{R} \). Define its decimal expansion of precision length \( k \) as follows: 
	- Let \( r_0 \) be the largest integer such that \( r_0 \leq x \).
	- Recursively, for \( i = 1 \) to \( k \), let \( r_i \) be the largest digit such that
	\[
	r_0 + \frac{r_1}{10} + \cdots + \frac{r_i}{10^i} \leq x.
	\]
	
	This gives the finite-precision decimal expansion \( r_0.r_1r_2\ldots r_k \) of \( x \). The sequence \( (r_0, r_1, \ldots, r_k) \) is the precision-\( k \) representation of \( x \). Note that this expansion is unique and may include zeros.
	
	For a point \( x = (x_1, x_2, \ldots, x_d) \in \mathbb{R}^d \), the finite-precision representation is obtained by applying the above procedure to each coordinate. While actual Turing machines use binary representations, we assume decimal precision for simplicity, noting that this does not affect computability theory: a number is computable if and only if its finite-precision representation can be generated to arbitrary precision by a Turing machine.
	
	\begin{remark}\label{gaprem}
		Let \( r_0, \ldots, r_k \) be the precision-\( k \) representation of a real number \( x \). Let \( \underline{x} \) be the number obtained by extending this expansion with zeros (i.e., \( r_{i > k} = 0 \)), and \( \bar{x} \) the number obtained by extending with nines (i.e., \( r_{i > k} = 9 \)). Then the interval \( [\underline{x}, \bar{x}] \) contains all real numbers that share the same decimal prefix. The length of this interval is \( \epsilon = 10^{-k} \), and thus with precision length \( k \), we can distinguish numbers only at a resolution of at least \( \epsilon \).
	\end{remark}

	\subsection{The Oracle Model}
	
	We now describe the oracle model under which we analyze the computability of global optima.
	
	In the oracle model, the function \( f : C \to \mathbb{R} \), defined over a compact domain \( C \subset \mathbb{R}^d \), is not given explicitly. Instead, access to \( f \) is provided via an \emph{oracle}, which, given a query point \( x \in C \), returns the value \( f(x) \) up to a specified finite precision.
	
	Formally, we assume that the algorithm is implemented as a Turing machine that can query an oracle \( \mathcal{O}_f \), which returns a finite-precision approximation of \( f(x) \). That is, for a query \( x \in C \) and a requested precision \( \epsilon > 0 \), the oracle returns a rational number \( r \) such that
	\[
	|r - f(x)| < \epsilon.
	\]
	We assume that queries to the oracle include both the point \( x \in C \) (given in finite-precision form) and the desired output precision \( \epsilon = 10^{-k} \), for some integer \( k \geq 1 \). The oracle may internally compute \( f(x) \) using any method, but this process is opaque to the querying algorithm.
	
	\begin{definition}
		An algorithm is said to \emph{compute} the global minimum \( f(x^*) \) of a function \( f \) in the oracle model if, for any \( \epsilon > 0 \), it halts in finite time and outputs a point \( \hat{x} \in C \) such that
		\[
		|f(\hat{x}) - f(x^*)| < \epsilon.
		\]
	\end{definition}
	
	\begin{remark}
		Note that this model captures many real-world optimization scenarios, where the function \( f \) is treated as a black box — as in hyperparameter tuning, physical simulation, or learned models — and only finite-precision function evaluations are feasible.
	\end{remark}
	
	We emphasize that this model differs from classical computable analysis, where \( f \) itself is assumed to be a computable function. In the oracle model, the function may be adversarially defined with respect to the queries made by the algorithm, and only pointwise access is allowed.
	
	
	\subsection{Problem Definition}
	
	We consider the problem of approximating global optima of continuous functions using a Turing machine with access to a function oracle.
	
	The oracle provides finite-precision evaluations of a continuous function \( f : C \to \mathbb{R} \), where \( C \subset \mathbb{R}^d \) is a compact domain. Specifically, for any finite-precision input \( x \in C \) and any requested output precision \( \epsilon > 0 \), the oracle returns a rational approximation \( r \) such that
	\[
	|r - f(x)| < \epsilon.
	\]
	The Turing machine is equipped with three tapes: one for computation, one for querying the oracle (including the input point and required precision), and one for reading and storing the oracle's responses \cite{sorbook}. The third tape may also store the history of queried points \( x_0, x_1, \ldots, x_k \) and the corresponding function values \( f(x_0), f(x_1), \ldots, f(x_k) \), which can be used to generate further queries.
	
	The objective is to compute an approximation to the global minimum value of \( f \), or a point achieving such a value. Let \( x^* \in C \) be a global minimizer, i.e., \( f(x^*) \leq f(x) \) for all \( x \in C \). The goal is to output a point \( x_o \in C \), of specified finite precision, such that
	\[
	|f(x_o) - f(x^*)| < \epsilon.
	\]
	We show that, in general, this problem is \emph{not computable} in the oracle model.
	
	We now formally define the problems considered in this paper:
	
	\begin{problem}[Global Minimum Value]
		Given a continuous, non-convex function \( f \), is there a Turing machine with access to the function oracle that can compute the global minimum value \( f(x^*) \)?
	\end{problem}
	
	\begin{problem}[Global Minimizer]
		Given a continuous, non-convex function \( f \), is there a Turing machine with access to the function oracle that can compute the global minimizer \( x^* \)?
	\end{problem}
	
	\begin{problem}[\( \epsilon \)-Approximation]
		Given a continuous, non-convex function \( f \) and an input precision \( \epsilon > 0 \), is there a Turing machine with access to the function oracle that can compute a point \( x_o \in C \) such that \( |f(x_o) - f(x^*)| < \epsilon \)?
	\end{problem}
	
	In the subsequent sections, we answer all of the above questions in the negative.
	
	\section{Main Theorem}
	
	The following results offer a formal foundation for the empirical difficulty encountered in global optimization tasks. We show that, even in the absence of algorithmic constraints such as runtime, the problem is inherently undecidable in the oracle model.
	
	Let \( f : C \to \mathbb{R} \) be a continuous function on a compact domain \( C \subset \mathbb{R}^d \). Let the set of global minimizers be denoted by \( G^f \). We begin by establishing that for any desired precision, there exists a finite-precision approximation to the global minimum.
	
	\begin{lemma}\label{preclem}
		For all \( \epsilon > 0 \), there exists a finite-precision point \( x^*_\epsilon \) such that
		\[
		|f(x^*_\epsilon) - f(x^*)| < \epsilon,
		\]
		where \( x^* \in G^f \) is a global minimizer.
	\end{lemma}
	
	\begin{proof}
		By continuity of \( f \), for any \( \epsilon > 0 \), there exists \( \delta > 0 \) such that \( |x - y| < \delta \Rightarrow |f(x) - f(y)| < \epsilon \). Let \( x^* \) be a global minimizer of \( f \), and let \( x^*_\epsilon \) be its finite-precision representation at precision length \( n \), with point spacing less than \( \delta \) (see Remark~\ref{gaprem}). Then \( |x^*_\epsilon - x^*| < \delta \Rightarrow |f(x^*_\epsilon) - f(x^*)| < \epsilon \).
	\end{proof}
	
	\begin{definition}
		Let \( G^f_{\epsilon,k} \) denote the set of all finite-precision points of length \( k \) whose function value is within \( \epsilon \) of the global minimum. Let \( G^f_\epsilon = \bigcup_{k} G^f_{\epsilon,k} \).
	\end{definition}
	
	Since there are finitely many points of given finite precision, each set \( G^f_{\epsilon,k} \) is finite. To simplify our analysis, we assume the global minimizer is unique, i.e., \( G^f = \{x^*\} \). This assumption is common in optimization theory, particularly for strictly convex functions.
	
	We now present a sequence of supporting lemmas:
	
	\begin{lemma}\label{checkhalt}
		The problem of deciding whether a real-valued function \( g \) is identically zero is undecidable.
	\end{lemma}
	
	\begin{proof}
		Suppose a Turing machine queries points \( x_1, x_2, \ldots \) and receives that \( g(x_i) = 0 \) for all \( i \). Let \( g' \) be a function that agrees with \( g \) on all queried points but satisfies \( g'(x) \neq 0 \) at some unqueried \( x \). Then no algorithm can distinguish \( g \) from \( g' \). Hence, the identity test is undecidable.
	\end{proof}
	
	\begin{lemma}\label{checklem}
		There is no algorithm to decide whether a given point \( x_k \in C \) is an \( \epsilon \)-approximation to the global minimum of \( f \).
	\end{lemma}
	
	\begin{proof}
		Let \( f \) be a continuous function such that \( f(x) \leq 0 \) for all \( x \in C \). Then deciding whether \( f \) is identically zero is equivalent to checking whether the global minimum value is zero. Define \( f'(x) := \max\{0, f(x) + \epsilon\} \). Then \( f' \) is identically zero if and only if \( f \) is \( \epsilon \)-close to zero everywhere. Hence, deciding whether a point is \( \epsilon \)-optimal reduces to deciding whether a function is identically zero — which is undecidable by Lemma~\ref{checkhalt}.
	\end{proof}
	
	\begin{lemma}\label{optlem}
		There is no algorithm to decide whether a function value \( f(x_k) \) is within \( \epsilon \) of the global minimum.
	\end{lemma}
	
	\begin{proof}
		The argument parallels Lemma~\ref{checklem}. Since function values are only accessible via the oracle, testing whether \( f(x_k) \) is close to \( f(x^*) \) is equivalent to testing whether \( x_k \) is \( \epsilon \)-optimal.
	\end{proof}
	
	\begin{theorem}[Main Result]\label{mainthm}
		Let \( f : C \to \mathbb{R} \) be a continuous, non-convex function on a compact domain, and assume access to \( f \) only via a function oracle. Then there is no algorithm that can compute an \( \epsilon \)-approximate minimizer or optimal value of \( f \).
	\end{theorem}
	
	\begin{proof}
		From Lemma~\ref{preclem}, we know there exists a point \( x^*_\epsilon \) of finite precision such that \( |f(x^*_\epsilon) - f(x^*)| < \epsilon \). Suppose there exists an algorithm that can compute some point \( x_k \) such that \( |f(x_k) - f(x^*)| < \epsilon/2 \). Then we could check whether any new point \( x \) satisfies \( |f(x) - f(x_k)| < \epsilon/2 \), thereby deciding whether \( x \) is \( \epsilon \)-optimal — which contradicts Lemma~\ref{checklem}. The same contradiction applies to the function value, completing the proof.
	\end{proof}
	
	\begin{corollary}
		It is undecidable whether a given local minimum is a global minimum.
	\end{corollary}
	
	\begin{remark}
		Even with access to higher-order oracles (e.g., for gradients or Hessians), the reduction from zero-testing still applies. Hence, global optimization remains non-computable in such cases.
	\end{remark}
	
	\begin{remark}
		Our oracle model is distinct from the setting of computable analysis, where both the function and reals are assumed to be computable (see \cite{pour}). Here, only pointwise oracle access is available.
	\end{remark}
	
	\begin{remark}
		The finite-precision model is not restrictive. If a real solution can be computed, its finite-precision truncation yields a valid \( \epsilon \)-approximate solution.
	\end{remark}
	
	\begin{remark}
		This result does not apply to local minima, as the (negative) function need not be zero if its local optima is zero.
	\end{remark}

	\section{Characterizing Global Optima Computability}
	
	In this section, we characterize when the global optimum of a function is computable using the notion of a predicate that encapsulates an approximate decision rule over the domain.
	
	\begin{definition}
		Let \( f : C \rightarrow \mathbb{R} \) be a continuous function. Define the binary predicate \( P^f(x, y) \) as:
		\[
		P^f(x, y) = 
		\begin{cases}
			\text{True} & \text{if } f(x) \leq f(y), \\
			\text{False} & \text{otherwise}.
		\end{cases}
		\]
	\end{definition}
	
	\begin{definition}
		Let \( Q(\zeta, x, y) \) be a computable ternary predicate. We say \( P^f \subset Q \) if there exists a real number \( \zeta \) such that for all \( x, y \in C \),
		\[
		P^f(x, y) = Q(\zeta, x, y).
		\]
	\end{definition}
	
	\begin{definition}
		We say the predicate \( Q(\zeta, x, y) \) is \emph{true} for a function \( f \) if there exists \( \zeta \in \mathbb{R} \) such that it holds for all \( x, y \in C \). It is \emph{computable} if, for any input \( x, y \), the value of \( Q(\zeta, x, y) \) can be evaluated by a Turing machine given \( \zeta \).
	\end{definition}
	
	We now prove that such a predicate \( Q \) provides a necessary and sufficient condition for the computability of global optima.
	
	\begin{theorem}\label{gloprop}
		The global optimum (value or minimizer) of a function \( f \) can be approximated to any desired accuracy if and only if there exists a computable predicate \( Q(\zeta, x, y) \) such that \( P^f \subset Q \) and \( Q \) is true.
	\end{theorem}
	
	\begin{proof}
		(\emph{If} direction.) Suppose such a computable predicate \( Q \) exists. Then there exists a computable \( \zeta \) such that for all \( x, y \in C \), \( P^f(x, y) = Q(\zeta, x, y) \). Let \( x_0 \in C \) be arbitrary. Since \( Q \) is computable, we can compute whether any candidate point \( x \) satisfies \( Q(\zeta, x, x_0) \). Define a sequence \( x_{k+1} \) such that \( Q(\zeta, x_{k+1}, x_k) \) holds. Because \( P^f \subset Q \), this implies \( f(x_{k+1}) \leq f(x_k) \), and hence \( \{f(x_k)\} \) is a decreasing sequence bounded below. By compactness, \( \{x_k\} \) has a convergent subsequence whose limit is a global minimizer. Since each \( x_k \) is computable (finite-precision), the limit is effectively approximable.
		
		(\emph{Only if} direction.) Assume the global minimizer \( x^* \) is computable. Then \( \zeta := \|x^*\| \) is also computable. Define \( Q(\zeta, x, y) := f(x^*) \leq f(y) \). This predicate is clearly computable and true, and satisfies \( P^f \subset Q \).
	\end{proof}
	
	\begin{remark}
		This result gives a semantic condition for computability: global optima can be approximated if the dominance relation \( f(x) \leq f(y) \) can be embedded in a computable predicate parameterized by a real number \( \zeta \). The predicate acts as a global certificate or guide for optima.
	\end{remark}
	
	\begin{remark}
		Lipschitz continuity is one such property. Define
		\[
		Q(L, x, y) := |f(x) - f(y)| \leq L \|x - y\|.
		\]
		Here, \( \zeta = L \), the Lipschitz constant. If this constant (or an upper bound) is known, then it is known from the literature (e.g., \cite{nest}, Theorem 1.1.2) that the global minimum can be approximated to arbitrary accuracy. Similarly, known bounds on derivatives or gradient norms can serve as computable global properties.
	\end{remark}

	\section{An Algorithm with Known Basin of Attraction}
	
		In this section, we also assume the function $f$ to be differentiable. Let us denote the gradient by $\triangledown f(x)$. The algorithm takes as input the lower bound $m$ on the basin of attraction of the global minimizer. By basin of attraction we mean the following: if we let the initial point to be in the hypercube of length $m$ in all co-ordinates, i.e., in the basin of attraction around the global minima $B_m(x^*)$ then the gradient descent algorithm will converge to the minima. This is another example of a global optima property (independent of the last argument), 
	\[Q(m,x,\cdot) := \mbox{ If }x \in B_m(x^*) \mbox{ and } x \neq x^* \mbox{ then } \triangledown f(x) \neq 0 .\]
	Note that there exists a $m$ such that $Q(m,x,\cdot)$ is True. If this basin of attraction size $m$ can be computed or its lower bound known then we can compute the global minima (Theorem \ref{gloprop}) to any approximation. In this section, we given an algorithm which converges to the global minima if this basin length $m$ or its lower bound is known.
	
	The algorithm finds the point $z_k$ where the function takes a minimum amongst all points at a distance of $m$ from each other and does a gradient descent step from the point $z_k$. The algorithm outputs an $\epsilon$-approximation to the global optima (see Theorem \ref{linconvthm} for the value of $\epsilon$). In this algorithm for simplicity, we do not consider line searches and use constant step-size $t > 0$. The figure \ref{figalg} shows the gradient descent step taken at the point which has the minimum function value amongst all the points in the grid.
	
	We note the similarity of our algorithm with the one considered in paper \cite{helon}, where the basin of attraction of global minimizer is first found by searching then a gradient descent is performed. In our algorithm, these two steps are interleaved. The major issue with their algorithm is that they assume the value of the global minima is known which they assume to be zero. But this need not be known in real-world problems. This assumption is not needed with our approach. Moreover, we have formally shown the convergence of our algorithm.
		
	\begin{algorithm}[ht]
		\caption{Global Optimization Algorithm}
		\textbf{Input:} Oracle access to function \( f \), and a lower bound \( m \) on the side length of a hypercube fully contained within the basin of attraction of the global minimizer. \\
		\textbf{Output:} An \( \epsilon \)-approximation to the global minimizer.
		\hrule
		\begin{algorithmic}[1]
			\STATE Let \( C = [a, b]^d \). For simplicity, assume all coordinates share the same interval.
			\STATE Set \( y_0 = [a, a, \dots, a] \in \mathbb{R}^d \).
			\STATE Generate grid points \( y_j = y_{j-1} + m \), for \( j = 1, \ldots, (b - a)/m \) in each dimension.
			\STATE Let \( z_0 = \arg\min_j \{ f(y_j) \} \), and set \( x_0 := z_0 \).
			\FOR{$k = 1, \ldots, \mathcal{L}$}
			\STATE Define the next grid anchor \( y_0 = x_{k-1} - t \nabla f(z_{k-1}) \).
			\STATE Generate grid points \( y_j = y_{j-1} + m \), for \( j = 1, \ldots, (b - a)/m \).
			\STATE Let \( z_k = \arg\min_j \{ f(y_j) \} \).
			\STATE Set \( x_k = z_k - t \nabla f(z_k) \).
			\ENDFOR
			\RETURN \( x_k \)
		\end{algorithmic}
	\end{algorithm}
		
	\begin{figure}[!ht] 
		\centering
		\begin{tikzpicture}
			\begin{axis}[
				axis lines = left,
				xticklabels = \empty,
				yticklabels = \empty,
				width = 0.55\textwidth,
				xlabel = $x$,
				ylabel = {$f(x)$},
				ymax = 40,
				ymin = 5,
				grid = both,
				ticks = both,
				minor x tick num = 2,
				axis lines=middle,
				]
				
				\node[anchor=west] (source) at (axis cs:15,9.5){};
				\node (destination) at (axis cs:12.5,14){};
				\draw[->](destination)--(source);
				\node[anchor=west] at (destination) {Gradient descent step};
				\addplot[
				color=blue,
				mark=none,smooth,
				]
				coordinates {
					(0,23.1)(10,17.5)(20,8)(30,27.8)(40,34.6)(60,14.8)(80,18.8)(100,20)
				};
			\end{axis}
		\end{tikzpicture}
		\caption{The function $f$ to be minimized. Gradient descent step is shown for the interval where the function value is minimum. This interval is a subset of the basin of attraction of global minima.}
		\label{figalg}
	\end{figure}

	\section{Convergence Analysis}
	We show the convergence of the algorithm given in the preceding section. We make the following assumption.
	
	\begin{assumption}\label{ass1}
		The function $f$ is twice differentiable. The gradient of $f$ is Lipschitz continuous with constant $0 < L < 1$, i.e., \[ \parallel \triangledown f(x) -\triangledown f(z) \parallel_2 \leq L \parallel x -z \parallel_2. \]
		That is we have $\triangledown^2 f(x) \preceq LI.$
	\end{assumption}
	
	We first state the following lemma used in the proof of the convergence theorem.
	
	\begin{lemma} \label{boundlem}
		Assume that the function $f$ satisfies Assumption \ref{ass1} and the step-size $t \leq 1/L$. We also assume that the global minima $x^*$ is unique. Then there exists a constant $R > 0$ such that for all balls $B(x^*,r)$ with radius $r < R$, there is a $M_{r} > 0$ and that the iterates of the algorithm $\{x_k\}$ remains in this ball $B(x^*,r)$ asymptotically, i.e., $x_k \in B(x^*,r)$ for $k \geq M_{r}$.
	\end{lemma}
	
	\begin{proof}
		From assumption \ref{ass1} we have that $\triangledown^2 f(x) - LI$ is negative semi-definite matrix. Using a quadratic expansion of $f$ around $f(x^*)$, we obtain the following inequality for $x \in B(x^*,r)$
		\begin{align}
			\nonumber f(x) &\leq f(x^*) + \triangledown f(x^*)^T (x-x^*) + \frac{1}{2} \triangledown^2 f(x^*) \parallel x-x^* \parallel_2^2 \\
			f(x) &\leq f(x^*) + \frac{1}{2} L \parallel x-x^* \parallel_2^2 \label{basbound}
		\end{align}
		Since $x^*$ is a global minima we have $f(x^*) \leq f(x)$ for all $x \in C$. Let $\tilde x$ be any local minima which is not global minima. Hence $f(\tilde x) =  f(x^*) + \delta_{\tilde x} $. Now let $\delta = \min_{\tilde x} \delta_{\tilde x}$. Since $\tilde x$ is local minima but not global minima we have $\delta > 0$. Take $R > 0$ such that for any $x \in B(x^*,R)$, 
		\[ \frac{L}{2}\parallel x - x^* \parallel_2^2 \leq \frac{\delta}{2} \] or that
		$R \leq \frac{\delta}{L}.$
		Now we have from equation \eqref{basbound}
		\[ f(x) \leq f(x^*) + \frac{\delta}{2}, \]
		for $x \in B(x^*,R)$. That is we have shown there exists a $R > 0$ such that for any $x \in B(x^*,R)$, 
		\begin{equation}
			f(x) \leq f(\tilde x). \label{ballremain}
		\end{equation}
		Now we observe the following:
		\begin{enumerate}
			\item from equation \eqref{ballremain} we can see that no other local minima can have a value $f(\tilde x)$, lower than the function value in this ball $B(x^*,R)$
			\item for sufficiently small step-size $t \leq 1/L$, the function value decreases with each gradient step (see equation \eqref{fybound} in proof of Theorem \ref{linconvthm})
		\end{enumerate}
		That is if $x \in B(x^*,R)$, the iterates in the algorithm can not move to another hypercube around some local minima $\tilde x$. Or that for all $r < R$ there exists $M_{r} > 0$ such that for $k \geq M_{r}$ the iterates remain in the ball $B(x^*,r)$ around $x^*$. 
	\end{proof}
	
	\begin{theorem}
		Let $x^*$ be the unique global minimizer of the function $f$. We have for the iterates $\{x_k \}$ generated by the algorithm \[ \lim_{k \rightarrow \infty} f(x_k) = f(x^*). \]
	\end{theorem}
	
	\begin{proof}
		Now from Lemma \ref{boundlem} we have $R > 0$ such that for all $r < R$ there exists $M_{r} > 0$ with $x_k \in B(x^*,r)$ for $k \geq M_{r}$. From the algorithm we also know that the function value decreases with each iteration. Thus we see that the sequence $\{f(x_k)\}$ converges as it is monotonic and bounded. Take a sufficiently small $r < R$, such that $B(x^*,r)$ lies in the basin of attraction. Hence we also have that $\lim_{k \rightarrow \infty} f(x_k) = f(x^*)$ as in the basin of attraction around the global minima the gradient descent converges to the minima.
	\end{proof}
	
	\begin{theorem}\label{linconvthm}
		Let $x^*$ be the unique global minimizer of the function $f$. For simplicity denote $M = M_{r}$. Let step-size $t \leq 1/L$ where  $L$ is Lipschitz constant of the gradient function in Assumption \ref{ass1}. If we also assume that the function is convex in the ball $B(x^*,r)$ we can show that at iteration $k > M$, $f(x_k)$ satisfies
		\[ f(x_k) - f(x^*) \leq \frac{\parallel x_{M} - x^* \parallel_2^2}{2 t (M-k)} .\]
		That is the gradient descent algorithm converges with rate $O(1/k)$.
	\end{theorem}
	
	\begin{proof}
		Consider the gradient descent step $x_{k+1} = z_k - t \triangledown f(z_k)$ in the algorithm. Since the iterates remain in a ball around a global minima asymptotically, we have from Lemma \ref{boundlem} for $k \geq M_{r}, \,$ $z_k = x_k$. Now let $y = x - t \triangledown f(x)$, we then get:
		\begin{align*}
			f(y) &\leq f(x) + \triangledown f(x)^T (y-x) + \frac{1}{2} \triangledown^2 f(x) \parallel y-x \parallel_2^2 \\
			&\leq f(x) + \triangledown f(x)^T (y-x) + \frac{1}{2} L \parallel y-x \parallel_2^2 \\
			&= f(x) + \triangledown f(x)^T (x - t \triangledown f(x) - x) + \frac{1}{2} L \parallel y -x \parallel_2^2 \\
			&= f(x) -  t \parallel \triangledown f(x) \parallel_2^2 + \frac{1}{2} L \parallel  y - x\parallel_2^2 \\
			&= f(x) - \big(1-\frac{1}{2} Lt\big)t \parallel \triangledown f(x) \parallel_2^2.
		\end{align*}
		Using the fact that $t \leq 1/L$, $-\big(1-\frac{1}{2} Lt\big) \leq -\frac{1}{2}$, hence we have
		\begin{equation} \label{fybound}
			f(y) \leq f(x) - \frac{1}{2}t\parallel \triangledown f(x) \parallel_2^2.
		\end{equation}
		Next we bound $f(y)$ the objective function value at the next iteration in terms of $f(x^*)$. 
		Note that by assumption $f$ is convex in the ball $B(x^*,r)$. 
		Thus we have for $x\in B(x^*,r)$,
		\begin{align*}
			f(x) \leq f(x^*) + \triangledown f(x)^T (x-x^*)
		\end{align*}
		Plugging this into equation \eqref{fybound} we get,
		\begin{align*}
			f(y) &\leq f(x^*) + \triangledown f(x)^T (x - x^*) - \frac{t}{2} \parallel \triangledown f(x) \parallel_2^2 \\
			f(y) - f(x^*) &\leq \frac{1}{2t}\bigg( 2t\triangledown f(x)^T (x - x^*) - t^2 \parallel \triangledown f(x) \parallel_2^2 \bigg) \\
			f(y) - f(x^*) &\leq \frac{1}{2t}\bigg( 2t\triangledown f(x)^T (x - x^*) - t^2 \parallel \triangledown f(x) \parallel_2^2 \\
			&\quad \quad \quad \quad \quad \quad - \parallel x - x^* \parallel_2^2 + \parallel x - x^* \parallel_2^2 \bigg) \\
			f(y) - f(x^*) &\leq \frac{1}{2t}\bigg( \parallel x - x^* \parallel_2^2 - \parallel x - t\triangledown f(x) -  x^* \parallel_2^2 \bigg)
		\end{align*}
		By definition we have $y = x - t \triangledown f(x)$, plugging this into the previous equation we have
		\begin{equation}
			f(y) - f(x^*) \leq \frac{1}{2t}\bigg( \parallel x - x^* \parallel_2^2 - \parallel y -  x^* \parallel_2^2 \bigg) 
		\end{equation}
		This holds for all gradient descent iterations $i \geq M$. Summing over all such iterations we get:
		\begin{align*}
			\sum_{i=M}^k \big(f(x_i) - f(x^*)\big) &\leq \sum_{i=M}^k \frac{1}{2t} \bigg(\parallel x_{i-1} - x^* \parallel_2^2 - \parallel x_i - x^* \parallel_2^2 \bigg) \\
			&= \frac{1}{2t} \bigg(\parallel x_{M} - x^* \parallel_2^2 - \parallel x_k - x^* \parallel_2^2 \bigg) \\
			&\leq \frac{1}{2t} \bigg(\parallel x_{M} - x^* \parallel_2^2 \bigg).
		\end{align*}
		Finally using the fact that $f$ is decreasing in every iteration, we conclude that
		\[ f(x_k) - f(x^*) \leq \frac{1}{k} \sum_{i=M}^k \big(f(x_i) - f(x^*)\big) \leq \frac{1}{2t(M-k)} \parallel x_{M} - x^* \parallel_2^2 . \]
	\end{proof}
	
	\begin{remark}
		If the global minima $x^*$ is not unique, then the algorithm can oscillate around different minima. If we assume that the function is convex in a small interval around all these global minima, then we can show that the algorithm converges to one of the minimum points $x^*$. In addition like in the previous theorem we can also show that the convergence is linear.
	\end{remark}
	
	\begin{remark}
		We have not considered momentum based acceleration methods which fasten the rate of convergence in this paper.
	\end{remark}
	
	\subsection{Real-World Applications}
	
	Global optimization arises in a wide variety of domains, including science, engineering, finance, and machine learning. The following examples illustrate how our result — that global optima are not computable in the oracle model — applies to important real-world problems.
	
	\subsubsection{Portfolio Optimization}
	
	A classical problem in finance is to minimize the risk (typically measured by the variance of the portfolio's rate of return) subject to achieving a specified expected return. In the simplest setting, this is a convex optimization problem and can be solved efficiently. However, when we introduce realistic extensions such as transaction costs, taxes, and market impact, the resulting objective becomes non-convex, often with many local minima. Our result implies that even approximate solutions to such extended problems are not computable in the general oracle setting.
	
	\subsubsection{Supervised Learning}
	
	In supervised learning, particularly in training neural networks (NNs), the loss function is often highly non-convex and rugged. For instance, in image classification tasks, such as recognizing handwritten digits, the loss surface has many local optima. Although backpropagation — a gradient-based method — is commonly used, our result implies that it cannot, in general, compute the global optimum or even an $\epsilon$-approximation to it.
	
	Consider a concrete example: classifying handwritten digits \cite{bishop}. We are given a training set $D_T = \{(x_i, C(x_i))\}$ of images $x_i$ labeled with digits $C(x_i) \in \{0, \ldots, 9\}$. The goal is to learn a model that generalizes to unseen images in a test set $D_E$. A typical approach is to fit a function $y(x) = f(x, W)$, where $W$ are the model parameters and $f$ is a non-linear activation function. While closed-form solutions are available under strong (and often unrealistic) assumptions, such models underperform on real-world tasks. More flexible models like NNs or Support Vector Machines perform better but rely on heuristics for optimization due to the non-convexity. Our result provides theoretical justification: no algorithm can guarantee convergence to the global optimum.
	
	\subsubsection{Chemical Process Optimization}
	
	Chemical process optimization is a vast field involving problems where complex and often non-analytic cost functions must be minimized \cite{chembook}. Consider a scenario where a manufacturer must distribute a chemical product to various customers from multiple plant locations. The problem involves determining how much to produce at each plant ($Y_i$) and how much to send to each customer ($Y_{ij}$), such that the total cost — including production, transportation, and capacity-related costs — is minimized.
	
	If even one component of the objective (e.g., cost or efficiency curve) is not analytically defined and must be evaluated through simulation or real-world measurements (i.e., via an oracle), our result implies that no algorithm can compute or approximate the global optimum of such problems.
	
	\vspace{1em}
	These examples highlight how global optimization, particularly in the oracle model, naturally arises in practice — and why its theoretical limitations must be taken seriously when designing algorithms for such problems.
		
	\section{Experimental Results}
	
	There exist many benchmark functions commonly used to evaluate global optimization algorithms. These include functions with multiple local minima, valley-shaped surfaces, and other challenging landscapes. One such function is the **Beale function**, defined on a rectangular domain $B = [(-4.5, -4.5), (4.5, 4.5)]$ as:
	
	\[
	f(x, y) = (1.5 - x + xy)^2 + (2.25 - x + xy^2)^2 + (2.625 - x + xy^3)^2.
	\]
	
	This function has a global minimum at $(3, 0.5)$ where $f(x^*) = 0$, and local minima elsewhere (e.g., at $(0,1)$). The gradient descent algorithm can get stuck in such local minima if not initialized near the global minimum.
	
	In our oracle model setting, the function $f(x)$ is not available analytically. Instead, it is accessible only through an oracle returning values for given inputs. According to our theoretical results, starting from an arbitrary point in the domain $B$, no algorithm (without additional information such as a bound on the basin of attraction) can compute or even approximate the global minimum. This holds even in the presence of higher-order oracles unless certain global properties (like bounds on derivatives) are known.
	
	\subsection*{Benchmark Tests}
	
	To empirically evaluate our algorithm, we implemented it on a suite of benchmark functions shown in Tables~\ref{benchf} and~\ref{benchfs}. For high-dimensional tests, the Rastrigin, Sphere, and Rosenbrock functions were evaluated in 20 dimensions. The algorithm's performance is measured by plotting the function value as a function of iterations. In all cases, we observe convergence to the known global optimum.
	
	\begin{table}[ht]
		\begin{small}
		\centering
		\caption{Benchmark Functions for Global Optimization}
		\label{benchf}
		\begin{tabular}{|c|c|}
			\hline
			\textbf{Name} & \textbf{Formula} \\
			\hline
			Rastrigin & $f(x) = An + \sum_{i=1}^n \left(x_i^2 - A \cos(2 \pi x_i)\right), \; A = 10$ \\
			\hline
			Ackley & $\begin{aligned}
				f(x, y) = &-20 \exp\left(-0.2 \sqrt{0.5(x^2 + y^2)}\right) \\
				& - \exp\left(0.5\left(\cos(2\pi x) + \cos(2\pi y)\right)\right) + e + 20
			\end{aligned}$ \\
			\hline
			Sphere & $f(x) = \sum_{i=1}^n x_i^2$ \\
			\hline
			Rosenbrock & $f(x) = \sum_{i=1}^{n-1} \left[100(x_{i+1} - x_i^2)^2 + (1 - x_i)^2\right]$ \\
			\hline
			Beale & $f(x, y) = (1.5 - x + xy)^2 + (2.25 - x + xy^2)^2 + (2.625 - x + xy^3)^2$ \\
			\hline
			Booth & $f(x, y) = (x + 2y - 7)^2 + (2x + y - 5)^2$ \\
			\hline
		\end{tabular}
			\end{small}
	\end{table}
	
	\begin{table}[ht]
		\centering
		\caption{Global Minima and Domains for Benchmark Functions}
		\label{benchfs}
		\begin{tabular}{|c|c|c|}
			\hline
			\textbf{Function} & \textbf{Global Minimum} & \textbf{Domain} \\
			\hline
			Rastrigin & $f(0,\ldots,0) = 0$ & $-5.12 \leq x_i \leq 5.12$ \\
			\hline
			Ackley & $f(0,0) = 0$ & $-5 \leq x, y \leq 5$ \\
			\hline
			Sphere & $f(0,\ldots,0) = 0$ & $\mathbb{R}^n$ \\
			\hline
			Rosenbrock & $f(1,\ldots,1) = 0$ & $\mathbb{R}^n$ \\
			\hline
			Beale & $f(3,0.5) = 0$ & $-4.5 \leq x, y \leq 4.5$ \\
			\hline
			Booth & $f(1,3) = 0$ & $-10 \leq x, y \leq 10$ \\
			\hline
		\end{tabular}
	\end{table}
	
	\begin{table}[ht]
		\centering
		\caption{Parameters Used in the Algorithm}
		\label{values}
		\begin{tabular}{|c|c|c|}
			\hline
			\textbf{Function} & \textbf{Step-size} & \textbf{Lower Bound on Basin Size} \\
			\hline
			Rastrigin & 0.0001 & 0.5 \\
			\hline
			Ackley & 0.0001 & 0.1 \\
			\hline
			Sphere & 0.001 & 0.3 \\
			\hline
			Rosenbrock & 0.001 & 0.5 \\
			\hline
			Beale & 0.0005 & 0.3 \\
			\hline
			Booth & 0.005 & 0.3 \\
			\hline
		\end{tabular}
	\end{table}
	
	These experiments demonstrate that, when the basin of attraction of the global optimum is known (or lower bounded), our algorithm effectively converges to the optimal value — in line with our theoretical guarantees.

	\begin{figure}[!ht]
		\centering
		\begin{subfigure}{.4\textwidth}
			\includegraphics[width=.6\textwidth,angle=270]{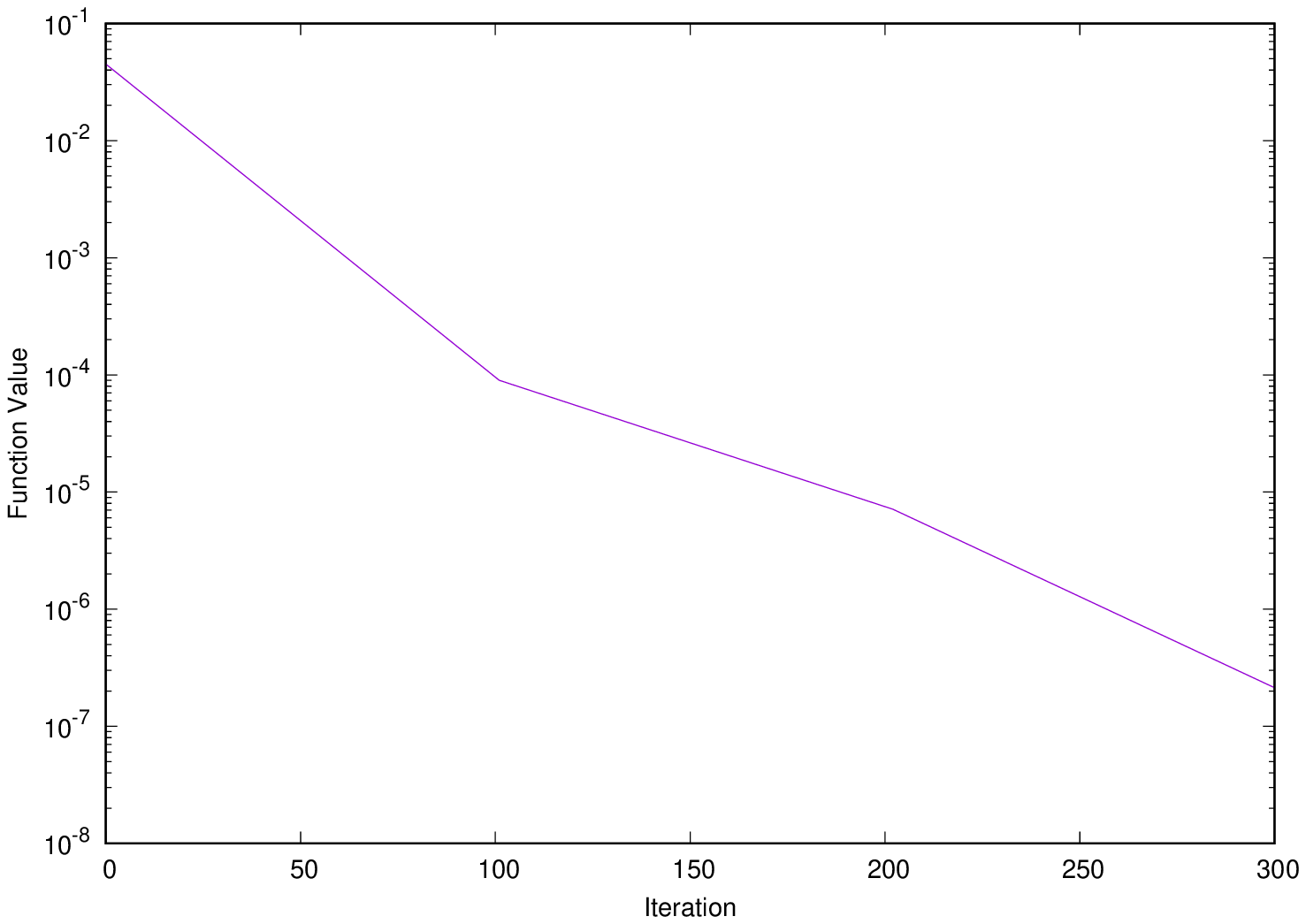}
		\end{subfigure}
		\begin{subfigure}{.4\textwidth}
			\centering
			\includegraphics[width=.7\textwidth,angle=270]{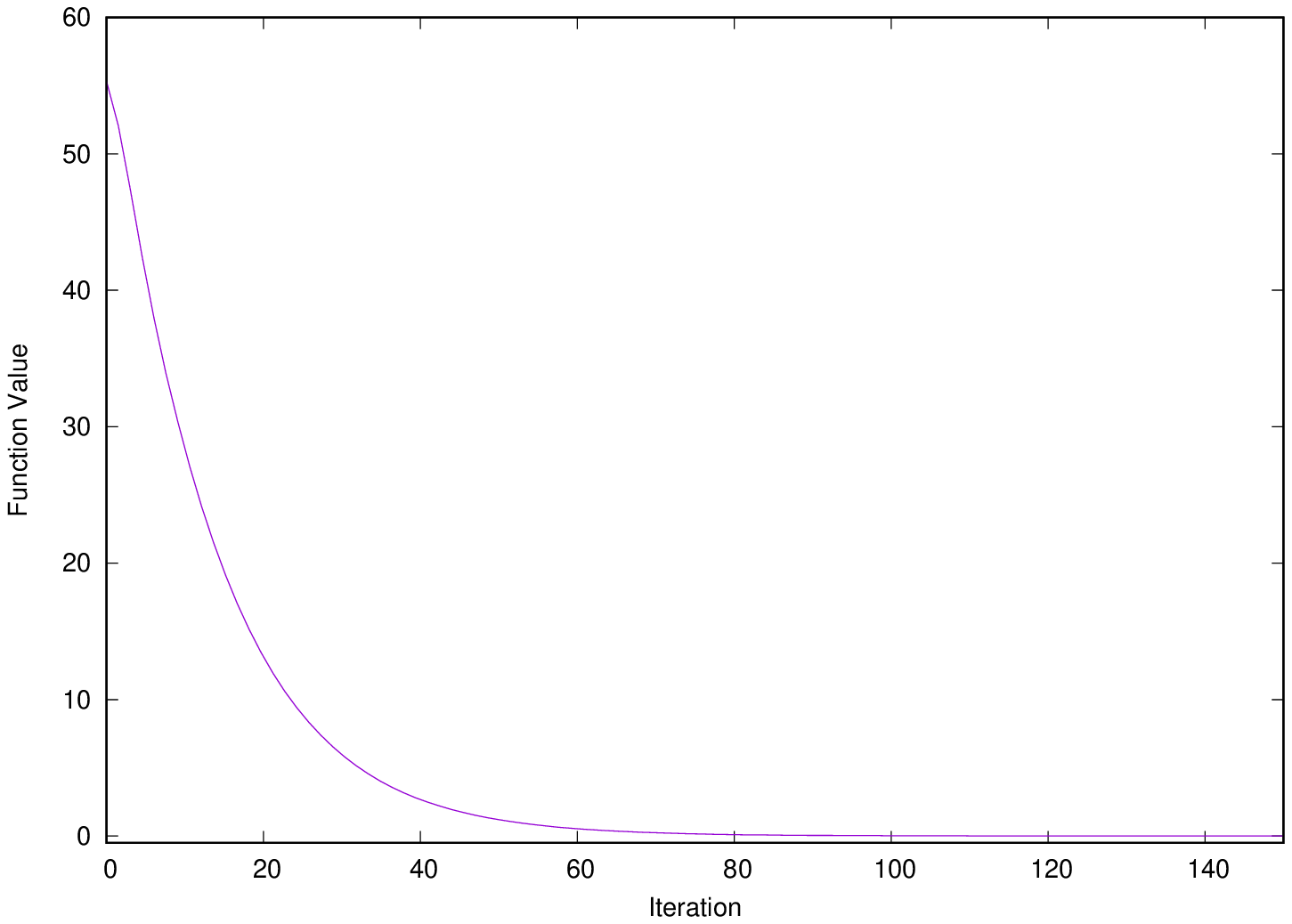}
		\end{subfigure}
		\caption{Convergence to Optimum for Ackley and Rastrigin Function}
	\end{figure}

	\begin{figure}[!ht]
		\centering
		\begin{subfigure}{.4\textwidth}
			\includegraphics[width=.7\textwidth,angle=270]{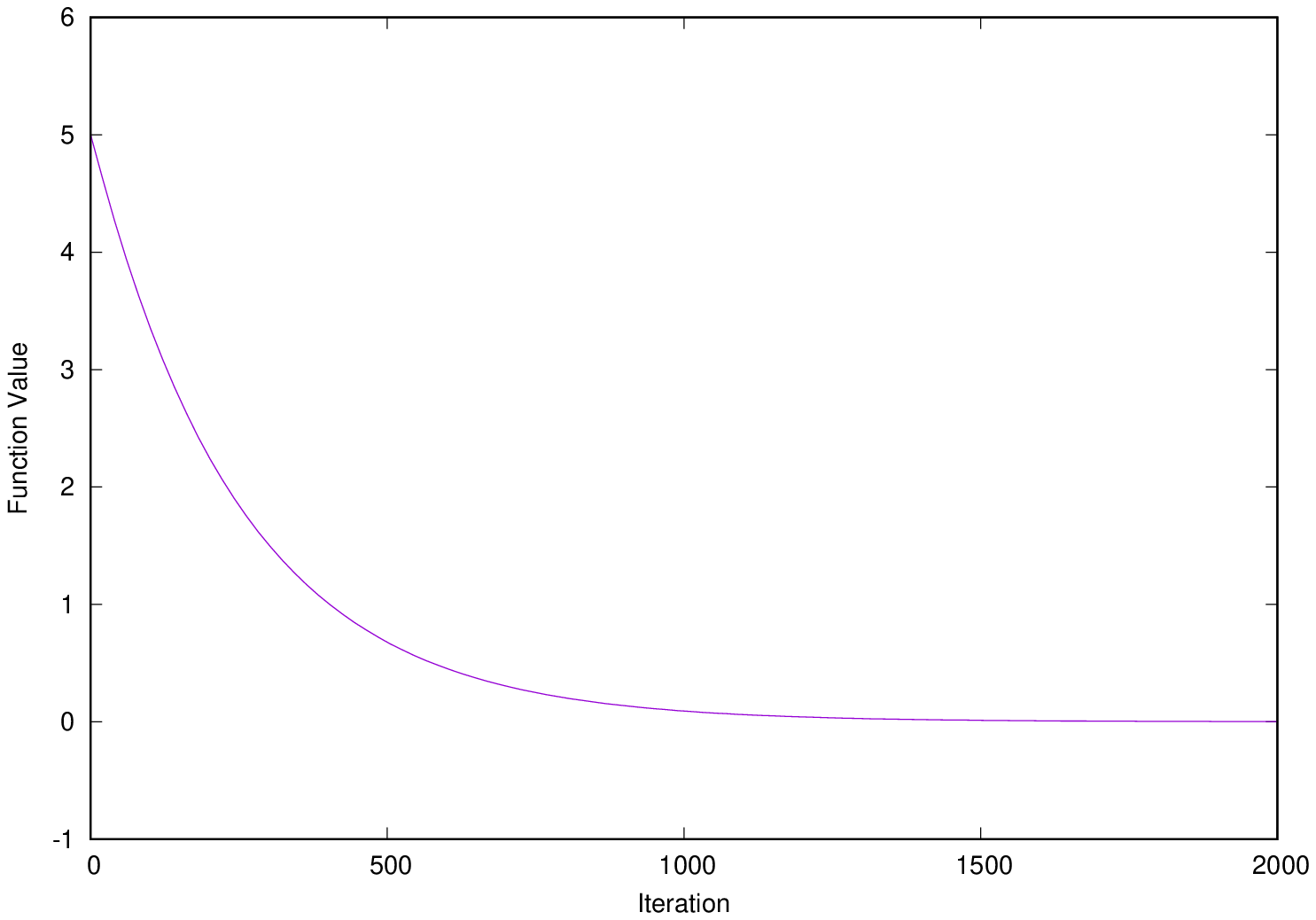}
		\end{subfigure}
		\begin{subfigure}{.4\textwidth}
			\centering
			\includegraphics[width=.6\textwidth,angle=270]{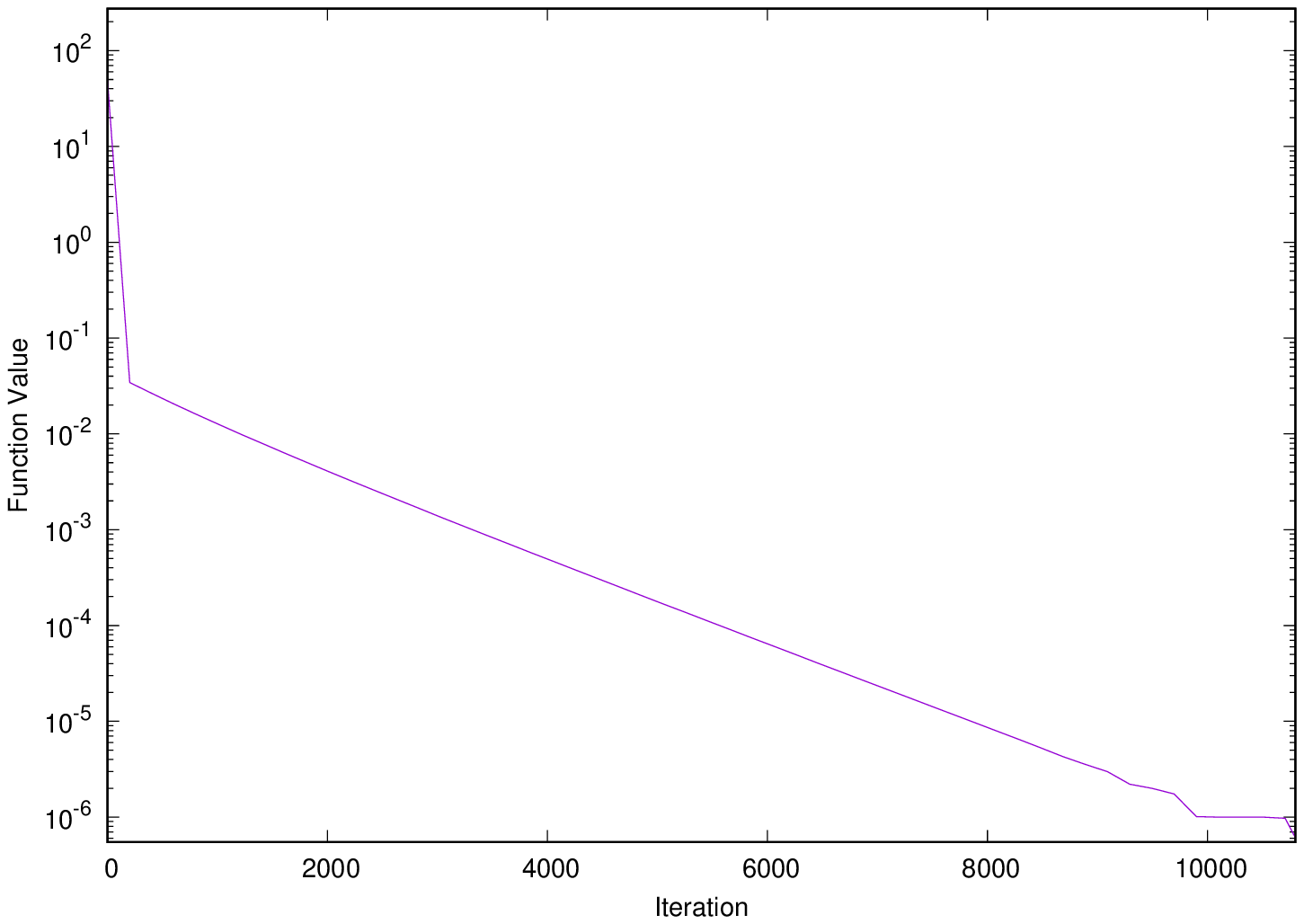}
		\end{subfigure}
		\caption{Convergence to Optimum for Sphere and Rosenbrock Function}
	\end{figure}

	\begin{figure}[!ht]
		\centering
		\begin{subfigure}{.4\textwidth}
			\includegraphics[width=.6\textwidth,angle=270]{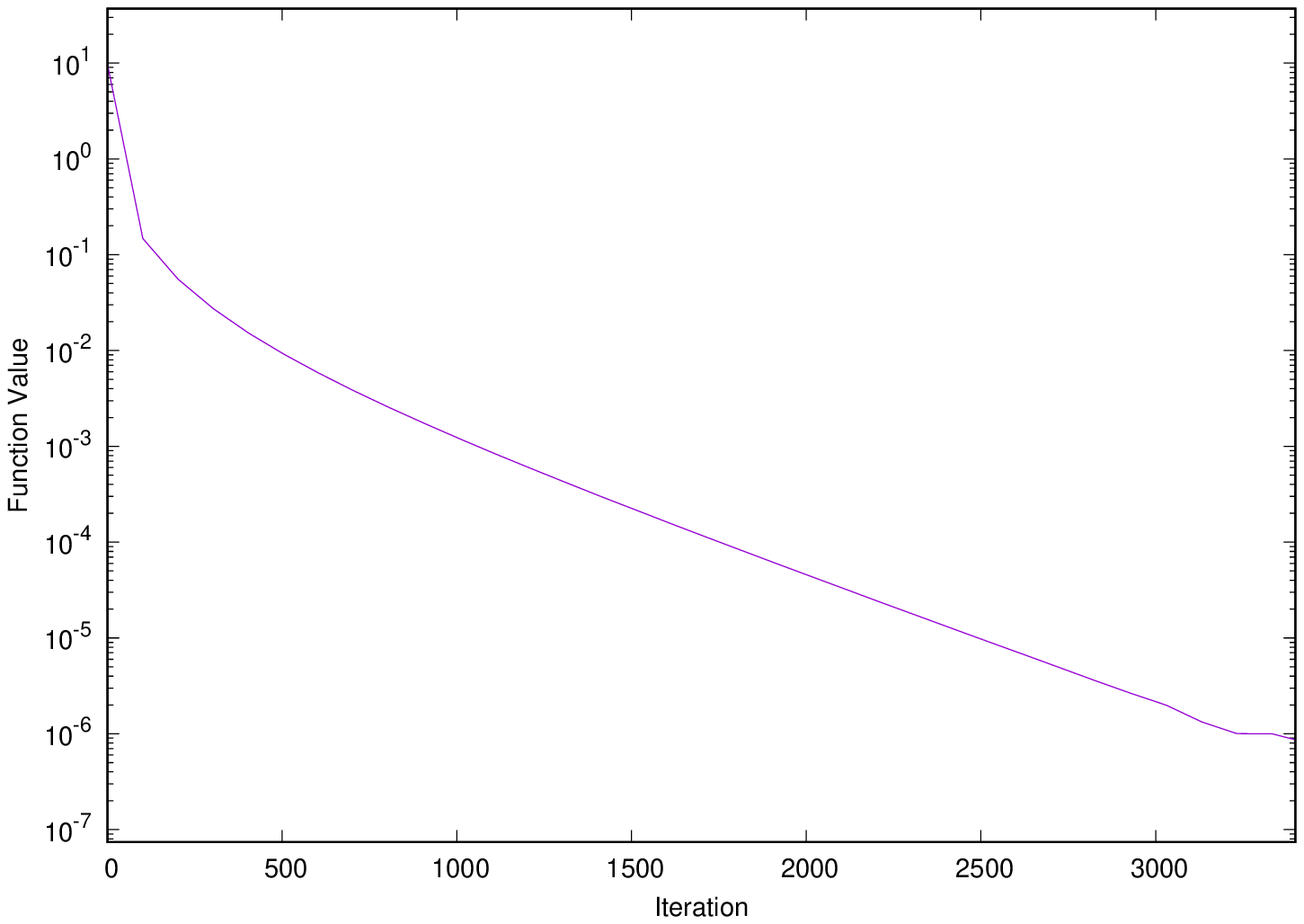}
		\end{subfigure}
		\begin{subfigure}{.4\textwidth}
			\centering
			\includegraphics[width=.6\textwidth,angle=270]{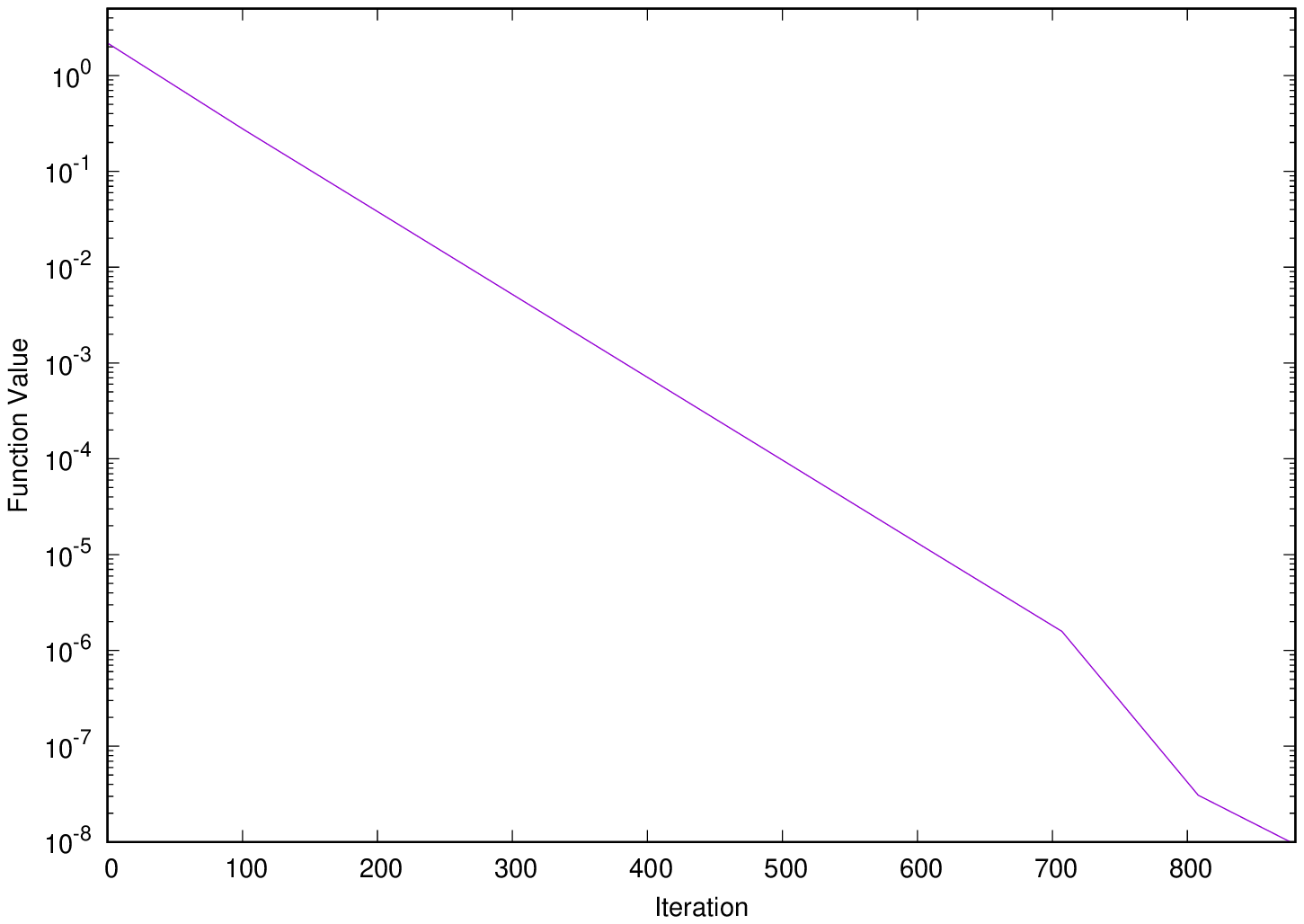}
		\end{subfigure}
		\caption{Convergence to Optimum for Beale and Booth Function}
	\end{figure}
	
	\section{Conclusion}
	We have demonstrated that in the oracle model, the problem of computing even approximate global optima of nonconvex functions is undecidable. This sharpens the classical understanding of optimization complexity and provides a rigorous explanation for the limitations of many real-world optimization strategies. Our characterization of computability also opens up a pathway for future work in defining structural conditions under which optimization becomes tractable or computable.

	
	To understand when global optimization becomes feasible, we introduced a necessary and sufficient condition based on a computable predicate property. We illustrated this through the concept of a known basin of attraction and provided an algorithm that successfully approximates the global minimum under such a condition. 
	
	Finally, we supported our theoretical findings with numerical experiments on standard benchmark functions, reinforcing the practical relevance of our results.

	\bibliographystyle{plain}
	\bibliography{glosuba}

\begin{thebibliography}{10}

\bibitem{bishop}
C.M. Bishop.
\newblock {\em Pattern Recognition and Machine Learning}.
\newblock Springer, 2006.

\bibitem{helon}
Cassius D’Helon, V~Protopopescu, Jack~C Wells, and Jacob Barhen.
\newblock Gmg—a guaranteed global optimization algorithm: Application to
  remote sensing.
\newblock {\em Mathematical and computer modelling}, 45(3-4):459--472, 2007.

\bibitem{chembook}
T.F. Edgar and D.M. Himmelblau.
\newblock {\em Optimization of Chemical Processes}.
\newblock McGraw Hill.

\bibitem{foster2019complexity}
Dylan~J Foster, Ayush Sekhari, Ohad Shamir, Nathan Srebro, Karthik Sridharan,
  and Blake Woodworth.
\newblock The complexity of making the gradient small in stochastic convex
  optimization.
\newblock In {\em Conference on Learning Theory}, pages 1319--1345. PMLR, 2019.

\bibitem{globook}
R.~Horst and H.~Tuy.
\newblock {\em Global Optimization: Deterministic Approaches}.
\newblock Springer-Verlag.

\bibitem{optmzer}
Y~Lee, H~Boche, and G~Kutyniok.
\newblock Computability of optimizers.
\newblock {\em IEEE Transactions on Information Theory}, 70(4):2967--2983,
  2023.

\bibitem{nph}
Katta~G Murty and Santosh~N Kabadi.
\newblock Some np-complete problems in quadratic and nonlinear programming.
\newblock {\em Mathematical Programming}, 39:117--129, 1987.

\bibitem{nest}
Y.~Nesterov.
\newblock {\em Introductory Lectures on Convex Optimization A Basic Course}.
\newblock Kluwer Academic Publishers.

\bibitem{pour}
M.~Pour-El and J.~Richards.
\newblock {\em Computability in analysis and physics}.
\newblock Springer, Heidelberg, 1989.

\bibitem{sorbook}
R.I. Soare.
\newblock {\em Turing Computability: Theory and Applications}.
\newblock Springer-Verlag.

\bibitem{zhang2020complexity}
Jingzhao Zhang, Hongzhou Lin, Stefanie Jegelka, Suvrit Sra, and Ali Jadbabaie.
\newblock Complexity of finding stationary points of nonconvex nonsmooth
  functions.
\newblock In {\em International Conference on Machine Learning}, pages
  11173--11182. PMLR, 2020.

\end{thebibliography}
	
	
	
	
	
	
	
\end{document}